\documentclass[11pt]{article}
\usepackage{amssymb,amsmath,amsthm,mathdesign}
\usepackage{pdfsync}
\usepackage{color}
\usepackage[colorlinks]{hyperref}

\newtheorem{theorem}{Theorem}[section]
\newtheorem{lemma}[theorem]{Lemma}
\newtheorem{proposition}[theorem]{Proposition}

\newtheorem{remark}[theorem]{Remark}

\numberwithin{equation}{section}

\newcommand{\be}{\begin{equation}}
\newcommand{\ee}{\end{equation}}

\newcommand{\bes}{\begin{equation*}}
\newcommand{\ees}{\end{equation*}}

\def\E{\bE}

\def\P{\bP} 

\def\bE{\mathbb{E}}

\newcommand{\R}{\mathbf{R}}

\renewcommand{\d}{{\rm d}}

\renewcommand{\geq}{\geqslant}
\renewcommand{\leq}{\leqslant}

\renewcommand{\P}{\mathrm{P}}

\def\m1{\mathbf{1}}

\allowdisplaybreaks[1]

\allowdisplaybreaks[1]

\title{Remarks on a fractional-time stochastic equation.}
\author{Mohammud Foondun\\University of Strathclyde }
\date{}
\begin{document}
\maketitle

\begin{abstract}
We consider a class of fractional time stochastic equation defined on a bounded domain and show that the presence of the time derivative induces a significant change in the qualitative behaviour of the solutions. This is in sharp contrast with the phenomenon showcased in \cite{FoonNual} and extented in \cite{Bin} and \cite{Eu}. We also show that as one {\it tunes off the fractional} in the fractional time derivative, the solution behaves more and more like its {\it usual} counterpart.

\end{abstract}
\section{Introduction and main results}
Consider the following stochastic heat equation on the interval $(0,L)$:
\begin{equation}\label{dirichlet}
\left|\begin{aligned}
&\partial_t v_t(x)=\frac{1}{2}\partial _{xx}v_t(x)+\lambda \sigma(v_t(x))\dot{W}(t,\,x)\;\; \text{for}\;\; 0<x<L\;\;\text{and}\; \;t>0\\
&v_t(0)=v_t(L)=0 \quad \text{for}\quad t>0,
\end{aligned}\right.
\end{equation}
where the initial condition $u_0:[0,L] \rightarrow \R_+$ is a non-negative bounded function with positive support inside the interval $[0,\,L]$. $\dot{W}$ denotes a space-time white noise and $\sigma:\R\rightarrow \R$ is a globally Lipschitz function satisfying $l_\sigma|x|\leq  |\sigma(x)| \leq L_\sigma|x|$ where $l_\sigma$ and $L_\sigma$ are positive constants. The positive parameter $\lambda$ is called {\it the level of the noise} and will play an important role in this paper.  We follow Walsh \cite{Walsh} to say that $v_t$ is a mild solution to \eqref{dirichlet} if \begin{equation}\label{mild:dirichlet}
v_t(x)=
(\mathcal{P}_Du_0)_t(x)+ \lambda \int_0^L\int_0^t p_D(t-s,\,x,\,y)\sigma(v_s(y))W(\d s\,,\d y),
\end{equation}
where $p_D(t,x,y)$ denotes the Dirichlet heat kernel, and
\begin{equation*}
(\mathcal{P}_D u_0)_t(x):=\int_0^L p_D(t,\,x,\,y) u_0(y)\,\d y.
\end{equation*}
In the above $p_D(t,\,x,\,y)$ is the heat kernel of the Dirichlet Laplacian and corresponds to the probability density function of Brownian motion killed upon exiting the domain $[0,\,L]$.  The proofs for existence-uniqueness and various other relevant technicalities can be found in \cite{Walsh} or \cite{minicourse}. We emphasise that the use of the subscript in $v_t$ and in other upcoming quantities indicates dependence on $t$ rather than derivative with respect to $t$. Throughout this paper, we will also assume that the spatial dimension is 1.

 In \cite{FoonNual}, it was shown that for small $\lambda$, the second moment decays exponentially fast while for large $\lambda$, the second moment grows exponentially fast. This was sharpened by using precise heat kernel estimates in \cite{Bin} and \cite{Eu}. A main aim of this note is to show that if one replaces the usual derivative by a fractional time derivative, this {\it phase transition} no longer holds and a more complicated picture emerges. From a practical point of view, our results are relevant since fractional time derivatives are often used in modelling memory in various systems. So knowing that the use of such derivatives can induce significant change in the qualitative properties of the solution is therefore very important. For results in the deterministic case, see \cite{SFMN} and references therein.

  Consider
\begin{equation}\label{dirichlet-2}
\left|\begin{aligned}
&\partial^\beta_t u_t(x)=\frac{1}{2}\partial _{xx}u_t(x)+I_t^{1-\beta}[\lambda \sigma(u_t(x))\dot{W}(t,\,x)]\;\; \text{for}\;\; 0<x<L\;\;\text{and}\; \;t>0\\
&u_t(0)=u_t(L)=0 \quad \text{for}\quad t>0,
\end{aligned}\right.
\end{equation}
where $\partial^\beta_t$ is the Caputo derivative with $\beta \in(0,\,1)$ and $I_t^{1-\beta}$ is the Riesz fractional integral operator given by 
\begin{equation*}
\partial^\beta_t f(t)=\frac{1}{\Gamma(1-\beta)}\int_0^t \frac{\partial
f(r)}{\partial r}\frac{\d r}{(t-r)^\beta}
\end{equation*}
and
\begin{equation*}
I^{1-\beta}_{t}f(t):=\frac{1}{\Gamma(1-\beta)}\int_{0}^{t}(t-\tau)^{-\beta}f(\tau)\d\tau
\end{equation*}
respectively.
All other parameters and conditions are as mentioned above.  We again use Walsh theory to make sense of the above equation via the following integral equation.
\begin{equation}\label{mild:dirichlet2}
u_t(x)=
(\mathcal{G}_Du_0)_t(x)+ \lambda \int_0^L\int_0^t G_D(t-s,\,x,\,y)\sigma(u_s(y))W(\d s\,,\d y),
\end{equation}
where $G_D(t,x,y)$ denotes the probability density function of the {\it slowed} killed Brownian motion associated with the fractional time operator. We will give more precise information about this later. The first term is defined by
\begin{equation*}
(\mathcal{G}_D u_0)_t(x):=\int_0^L G_D(t,\,x,\,y) u_0(y)\d y.
\end{equation*}
Since we will be studying the \eqref{mild:dirichlet2}, the precise definition of the operators $\partial^\beta_t$ and $I_t^{1-\beta}$ will not play any direct role. However, they are important in explaining the connection between \eqref{mild:dirichlet2} and \eqref{dirichlet-2}. 
We are now ready to state our first main result.
\begin{theorem}\label{Theo1}
Let $u_t$ denote the unique solution to \eqref{mild:dirichlet2}.  Then no matter what $\lambda$ is, the second moment of $u_t$ cannot decay exponentially fast. In fact, if we further assume that $\beta\in (0,\,\frac{1}{2}]$, then as $t$ gets large, $\sup_{x\in [0,\,L]} \E|u_t(x)|^2 $ grows exponentially fast for any $\lambda.$
\end{theorem}
Our results strongly rely on the representation given by \eqref{eigen} which are in terms of the Mittag-Leffler function denoted by $E_\beta(\cdot)$. More precisely, we will use the fact that for large times, the function $E_\beta(\cdot)$ behaves like $t^{-\beta}$ as opposed to the exponential decay when one looks at the usual time derivative. In fact, this observation was the starting point of this paper. This polynomial decay also explains the requirement that $\beta\in (0,\,\frac{1}{2}]$ which turns out to be sharp. The strategy behind the proof is quite simple. We will look at the second moment of the solution. Roughly speaking, this means that after taking the second moment,  \eqref{mild:dirichlet2} gives
\begin{align*}
\text{Second moment}=\text{Term 1}+ \text{Term 2}.
\end{align*}
`Term 1' will essentially have polynomial decay due the behaviour of the Mittag-Leffler function described above.  `Term 2' will also be in terms of these functions but more important will be depend on the second moment in a linear way. This is due to the fact that $\sigma$ is globally Lipschitz. The general thrust of our method is to obtain some kind of non-linear renewal inequalities for quantity which involve the second moment. While this strategy is not new and was used in \cite{FoonNual} and \cite{FK} among others, here the main winning observation is that `Term 1' decays polynomially despite the Dirichlet boundary condition. The following shows that the condition that $\beta\in (0,\,\frac{1}{2}]$ is not merely a technical limitation of the proof of Theorem \ref{Theo1}.
\begin{theorem}\label{Theo2}
Suppose that $\beta\in (\frac{1}{2},\,1)$. Then there exist $\lambda_l$ and $\lambda_u$ strictly positive such that for all $\lambda>\lambda_u$,  $\sup_{x\in [0,\,L]} \E|u_t(x)|^2 $ grows exponentially fast as time gets large and for $\lambda<\lambda_l$,  the quantity $\sup_{t>0}\sup_{x\in [0,\,L]} \E|u_t(x)|^2$ is finite.
\end{theorem}

The above two results show that for any fixed $\beta$, the solution to the stochastic heat equation \eqref{dirichlet-2} behaves very differently to that of \eqref{dirichlet}. One way to interpret the above result is that when the $\beta\in(0,\frac{1}{2}]$, the process $X_{E_t}^D$ defined in the next section do not reach the boundary fast enough and this allows the non-linear term to kick in. While for $\beta \in (\frac{1}{2},\,1)$, this process moves to the boundary fast enough so that the non-linear term do not have time to induce the exponential growth.

Theorem \ref{Theo2} also hints to the fact that as $\beta$ gets larger and larger, the solution to \eqref{dirichlet-2} behaves more and more like that of \eqref{dirichlet}.  This motivated the following result.

\begin{theorem}\label{Theo3}
Let $u_t$ and $u_t^{(\beta)}$ denote the solution to \eqref{dirichlet} and \eqref{dirichlet-2} respectively. The initial condition $u_0$ is the same for both equations. Then, for any $p\geq 2$, we have 
\begin{equation*}
\lim_{\beta\rightarrow 1}\sup_{x\in [0,\,L]}\E|u_t(x)- u_t^{(\beta)}(x)|^p=0.
\end{equation*}
\end{theorem}
The above theorem in conjunction with a continuity estimate give us a weak convergence result. Convergence of moments gives us convergence of finite dimensional distributions while the continuity estimate give tightness. We follow \cite{Bezdek} for some notations and ideas.  Fix $T>0$ and $N>0$ and set $\mathcal{C}=\mathcal{C}([0,\,T]\times [-N,\,N])$, the space of continuous functions on $[0,\,T]\times [-N,\,N]$ with the supremum norm. For any Borel set $A$ in $\mathcal{C}$, we take 
\begin{align*}
\P_\beta(A):=\begin{cases}
 \P(u^{(\beta)}\in A^o)&\quad \text{if} \quad \beta \in(0,\,1)\\
\P(u\in A^o)&\quad \text{if} \quad \beta=1,
\end{cases}
\end{align*}
where $$A^o:=\{f\in \mathcal{C}(\R_+\times \R): f \,\,\text{restricted to} \,\, [0,\,T]\times [-N,\,N]\,\, \text{is in}\,\,A\}.$$
Our final result is the following.
\begin{theorem}\label{Theo4}
As $\beta \rightarrow 1$, the measures $\P_\beta$ converge weakly to $\P_1$.
\end{theorem}
\section{Background Information}
The aim of this section is to gather some background information needed for the proofs of the main results. The reader can seek more precision in \cite{BLM} and \cite{CMN}. Let $X_t$ denote a Brownian motion and $X_t^D$ denote the Brownian motion killed upon exiting the domain $[0,\,L]$. It is well known that the following expansion holds for the probability density function of $X_t^D$,
\begin{equation*}
p_D(t,\,x,\,y):=\sum_{n=1}^\infty e^{-\lambda_n t}\phi_n(x)\phi_n(y).
\end{equation*}
Here $\lambda_n=\left(\frac{n\pi}{L}\right)^2$ are the eigenvalues of the Dirichlet Laplacian and $\phi_n(x):=\left(\frac{2}{L}\right)^{1/2}\sin\left(\frac{n\pi x}{L}\right)$ are the corresponding eigenfunctions. It is also known that
\begin{align*}
w_t(x)&:=\E^x[u_0(X_t^D)]\\
 &=\int_0^Lp_D(t,\,x,\,y)u_0(y)\,\d y
\end{align*}
solves the heat equation $\partial_t w_t(x)=\frac{1}{2}\partial _{xx}w_t(x)$ defined on $[0,\,L]$ with Dirichlet boundary condition and initial condition $u_0$. Now take $E_t$ to be the inverse of a stable subordinator of index $\beta\in (0,\,1)$. The process $X_{E_t}^D$ is just a time-changed of the killed Brownian motion $X^D_t$ and since $\beta\in (0,\,1)$, $X_{E_t}^D$ moves more slowly than $X_t^D.$  If we set $v_t(x):=\E^x[u_0(X_{E_t}^D)]$ then $v_t(x)$ solves the following equation,
\begin{equation}
\left|\begin{aligned}
&\partial^\beta_t v_t(x)=\frac{1}{2}\partial _{xx}v_t(x)\;\; \text{for}\;\; 0<x<L\;\;\text{and}\; \;t>0\\
&v_t(0)=v_t(L)=0 \quad \text{for}\quad t>0,
\end{aligned}\right.
\end{equation}
with initial condition $u_0$.  The probability density function of $X_{E_t}^D$ is denoted by $f_t(s)$ and satisfies the following relation
\begin{align}\label{relation}
f_t(s)=t\beta^{-1}s^{-1-1/\beta}g_\beta(ts^{-1/\beta}).
\end{align}
We have the following representation.
\begin{align*}
v_t(x)&:=\E^x[u_0(X_{E_t}^D)]\\
&=\int_0^{L}\int_0^\infty \sum_{n=1}^\infty e^{-\lambda_n s}\phi_n(x)\phi_n(y) f_t(s)\,\d su_0(y)\d y\\
&=\int_0^{L}\sum_{n=1}^\infty E_\beta(-\lambda_n t^\beta)\phi_n(x)\phi_n(y)u_0(y)\d y\\
&=\int_0^LG_D(t,\,x,\,y)u_0(y)\,\d y.
\end{align*}
We therefore have the following expansion,
\begin{equation}\label{eigen}
G_D(t,\,x,\,y):= \sum_{n=1}^\infty E_\beta(-\lambda_n t^\beta)\phi_n(x)\phi_n(y)
\end{equation}
which after a change of variable and using \eqref{relation} can also be written as 
\begin{align*}
G_D(t,\,x,\,y)&=\int_0^\infty \sum_{n=1}^\infty e^{-\lambda_n (t/u)^\beta}\phi_n(x)\phi_n(y) g_\beta(u)\d u\\
&=\int_0^\infty p_D(\left(\frac{t}{u}\right)^\beta,\,x,\,y)g_\beta(u)\,\d u.
\end{align*}
The Mittag-Leffler function $E_\beta(\cdot)$ which is the Laplace transform of $f_t(s)$ have the following property,
\begin{equation}\label{uniformbound}
\frac{1}{1 + \Gamma(1-\beta)x}\leq E_{\beta}(-x)\leq \frac{1}{1+\Gamma(1+\beta)^{-1}x} \ \  \ \text{for}\ x>0.
\end{equation}
This was proved by probabilistic means in \cite{Simon}. A simple consequence of the above is that 
\begin{equation*}
\sup_{x,\,y\in [0,\,L]}G_D(t,\,x,\,y)\lesssim \frac{1}{t^{\beta/2}}, 
\end{equation*}
for all $t>0$. Hence for small $t$, we have 
\begin{equation}\label{bound1}
\sup_{x,\,y\in [0,\,L]}G_D(t,\,x,\,y)\lesssim \frac{1}{t^{1/2}}, 
\end{equation}
and for large $t$, the left hand side is bounded. We will also need the following.

\begin{lemma}\label{der}
Fix $t>0$, then 
\begin{align*}
\sup_{x,\,y\in [0,\,L]}\frac{\partial G_D(t,\,x,\,y)}{\partial t}\lesssim \frac{1}{t^{1+\beta/2}}.
\end{align*}
\end{lemma}
\begin{proof}
We begin with the following representation
\begin{align*}
G_D(t,\,x,\,y)&=\sum_{n=1}^\infty\int_0^\infty e^{-\lambda_n s}\phi_n(x)\phi_n(y) f_t(s)\,\d s\\
&=\sum_{n=1}^\infty\int_0^\infty e^{-\lambda_n (\frac{t}{u})^\beta}\phi_n(x)\phi_n(y) g_\beta(u)\,\d u.
\end{align*}
This gives us
\begin{align*}
\frac{\partial G_D(t,\,x,\,y)}{\partial t}=\frac{1}{t}\sum_{n=1}^\infty\int_0^\infty -\lambda_n (\frac{t}{u})^\beta e^{-\lambda_n(\frac{t}{u})^\beta}\phi_n(x)\phi_n(y) g_\beta(u)\,\d u
\end{align*}
We thus have 
\begin{align*}
\left|\frac{\partial G_D(t,\,x,\,y)}{\partial t} \right|&\lesssim \frac{1}{t}\sum_{n=1}^\infty\int_0^\infty -\lambda_n (\frac{t}{u})^\beta e^{-\lambda_n(\frac{t}{u})^\beta}g_\beta(u)\,\d u\\
&\lesssim \frac{1}{t^{1+\beta/2}}.
\end{align*}
\end{proof}
To prove the following result will use the fact that
\begin{equation*}
\frac{\partial E_\beta(-\lambda_n t^\beta)}{\partial t}\lesssim \frac{\lambda_n t^{\beta-1}}{1+\lambda_n t^\beta}.
\end{equation*}
This can be found in \cite{kra}.
\begin{lemma}\
Fix $t>0$ and let $0<\eta<1$, then 
\begin{align}\label{bound2}
\sup_{x\in [0,\,L]}\int_0^t\int_0^L [&G_D(s,\,x+k,\,y)-G_D(s,\,x,\,y)]^2\,\d y\,\d s\leq C_1|k|^{\eta},
\end{align}
and if we further assume that $\eta<1-\beta/2$, then 
\begin{align}\label{bound3}
\sup_{x\in [0,\,L]}\int_0^t\int_0^L [G_D&(s+h,\,x,\,y)-G_D(s,\,x,\,y)]^2\,\d s\,\d y\leq C_2 |h|^{\eta},
\end{align}
where $C_1$ and $C_2$ are positive constants which can be taken to be independent of $\beta$.
\end{lemma}
\begin{proof}
The proof relies on the expansion of the Dirichlet heat kernel and the mean value theorem.  For the first part, we have
\begin{align*}
\int_0^t\int_0^L [&G_D(s,\,x+k,\,y)-G_D(s,\,x,\,y)]^2\,\d y\,\d s\\
&=\int_0^t \sum_{n=1}^\infty E_\beta(-\lambda_n s^\beta)^2[\phi_n(x+k)-\phi_n(x)]^2\,\d s\\
&\lesssim \int_0^t \sum_{n=1}^\infty 
E_\beta(-\lambda_n s^\beta)^2|\phi_n(x+k)-\phi_n(x)| \,\d s\\
&\lesssim  k^\eta\int_0^t \sum_{n=1}^\infty 
E_\beta(-\lambda_n s^\beta)^2 n^\eta \,\d s\\
&\lesssim k^\eta \int_0^t\frac{1}{s^{\beta(1+\eta)/2}}\,\d s.
\end{align*}
We note that the integral on the right hand side of the above display can be bounded by a quantity independent of $\beta$.  The second part follows from the following.
\begin{align*}
\int_0^t\int_0^L |G_D&(s+h,\,x,\,y)-G_D(s,\,x,\,y)|^2\,\d s\,\d h \\
&=\int_0^t\sum_{n=1}^\infty[E_\beta(-\lambda_n(s+h)^\beta)-E_\beta(-\lambda_ns^\beta)]^2|\phi_n(x)|^2\,\d s\\
&\lesssim \int_0^t\sum_{n=1}^\infty \frac{1}{1+\Gamma(1+\beta)^{-1}\lambda_ns^\beta}|\phi_n(x)|^2\left[\int_s^{s+h}\frac{1}{l}\,\,\d l\right]^{\eta}\d s\\
&\lesssim h^\eta\int_0^t\sum_{n=1}^\infty \frac{1}{1+\Gamma(1+\beta)^{-1}\lambda_ns^\beta}\frac{1}{s^\eta}\,\d s\\
&\lesssim h^\eta \int_0^t \frac{1}{s^{\beta/2}}\frac{1}{s^\eta}\,\d s.
\end{align*}
We take $\eta<1-\beta/2$ so that the integral above is defined. Morever it is bounded by a quantity which is independent of $\beta$.
\end{proof}
\begin{remark}
The above estimate can be sharpened a little bit. But this is sufficient for our needs.
\end{remark}
\begin{lemma}\label{con-kernel}
Fix $t>0$, then 
\begin{equation*}
\lim_{\beta \rightarrow 1}\sup_{x\in [0,\,L]}(\mathcal{G}_D u_0)_t(x)=(\mathcal{P}_D u_0)_t(x).
\end{equation*}
\end{lemma}
\begin{proof}
Given the upper bound on $G_D(t,\,x,\,y)$ and the fact that the initial condition is bounded above, it will be enough to show that that 
\begin{equation*}
\lim_{\beta\rightarrow 1}G_D(t,\,x,\,y)=p_D(t,\,x,\,y). 
\end{equation*}
This is straightforward. We use the Laplace transform of the Mittag-Leffler function to see that
\begin{equation*}
\int_0^\infty e^{-\theta t}E_\beta(-\lambda_n t^\beta)= \frac{\theta^{\beta-1}}{\theta^\beta-\lambda_n},
\end{equation*}
for any $\theta>0$.  We then  take limit as $\beta\rightarrow 1$ to conclude that 
\begin{equation*}
\lim_{\beta\rightarrow 1} E_\beta(-\lambda_n t^\beta)=e^{-\lambda_n t}.
\end{equation*}
Using the expansions for the heat kernel, we now have 
\begin{align*}
|G_D(t,\,x,\,y)-p_D(t,\,x,\,y)|&\leq \sum_{n=1}^\infty|E_\beta(-\lambda_nt^\beta)-e^{-\lambda_n t}|.
\end{align*}
We note that using the bounds on the Mittag-Leffler function, each term appearing in the summation can be bounded by a quantity independent of $\beta$ and summable.  We now take limit as $\beta\rightarrow 1$ on both sides to obtain the required result.
\end{proof}

\section{Proof of Theorems \ref{Theo1} and \ref{Theo2}}
For $\theta>0$, set
\begin{equation}\label{L}
\Lambda(\theta):=\int_0^\infty e^{-\theta t}E_\beta(-\lambda_1t^\beta)^2\,\d t.
\end{equation}
The important point here is to notice that one can use the bounds given by \eqref{uniformbound} to see that $\Lambda(\theta)$ tends to infinity as $\theta$ goes to zero if and only if  $2\beta\leq 1$.
\begin{proof}[Proof of Theorem \ref{Theo1}]

We prove the second statement of the theorem first.  We start with the mild formulation given by \eqref{mild:dirichlet2} which upon setting 
\begin{equation*}
\langle u_t, \phi_1\rangle:=\int_0^L u_t(x) \phi_1(x)\,\d x,
\end{equation*}
yields 
\begin{align*}
\langle u_t, \phi_1\rangle=E_\beta(-\lambda_1 t^\beta) \langle u_0, \phi_1\rangle +\lambda \int_0^t\int_0^L E_\beta(\lambda_1(t-s)^\beta)\phi_1(y)\sigma(u_s(y))W(\d s,\,\d y).
\end{align*}
We now take the second expectation and use the Ito-Walsh isometry to obtain
\begin{align}
\E\langle u_t, \phi_1\rangle^2&=E_\beta(-\lambda_1 t^\beta)^2 \langle u_0, \phi_1\rangle^2 +\lambda^2\int_0^t\int_0^L E_\beta(-\lambda_1(t-s)^\beta)^2\phi_1^2(y)\E|\sigma(u_s(y))|^2\d s\d y.\label{iso}
\end{align}
After some computations and using \eqref{L} together with the assumption $\sigma$ give us
\begin{align}\label{lower-renewal}
\int_0^\infty e^{-\theta t}\E\langle u_t, \phi_1\rangle^2\,\d t&=  \Lambda(\theta) \langle u_0, \phi_1\rangle^2+\lambda^2\Lambda(\theta)\int_0^\infty e^{-\theta t}\int_0^L\E|\sigma(u_s(y))|^2\phi_1^2(y)\d y\d t\nonumber \\
&\gtrsim  \Lambda(\theta) \langle u_0, \phi_1\rangle^2+\lambda^2l_\sigma^2\Lambda(\theta)\int_0^\infty e^{-\theta t}\E\langle u_t, \phi_1\rangle^2\,\d t\\
&\gtrsim  \Lambda(\theta) \langle u_0, \phi_1\rangle^2+2\int_0^\infty e^{-\theta t}\E\langle u_t, \phi_1\rangle^2\,\d t,\nonumber
\end{align}
where we have taken $\theta$ small enough to obtain the last inequality and used the fact that $\beta\in(0,\,\frac{1}{2}]$.  This means that for small enough $\theta$, 
\begin{equation*}
\int_0^\infty e^{-\theta t}\E\langle u_t, \phi_1\rangle^2\,\d t=\infty
\end{equation*}
which in turns means that for $t$ large enough, $\E\langle u_t, \phi_1\rangle^2$ grows exponentially. Now using the following, 
\begin{align*}
\E\langle u_t, \phi_1\rangle^2&\lesssim  \sup_{x\in [0,\,L]}\E|u_t(x)|^2,
\end{align*}
we can complete the proof of the second part of the theorem. The first part of the theorem merely follows from the fact that the second term of \eqref{iso} is positive and that the first term cannot have exponential decay.
\end{proof}

\begin{proof}[Proof of  Theorem \ref{Theo2}]
We use the same notation as in the proof of Theorem \ref{Theo1}. In fact \eqref{lower-renewal} also holds,
\begin{align*}
\int_0^\infty e^{-\theta t}\E\langle u_t, \phi_1\rangle^2\,\d t&\gtrsim  \Lambda(\theta) \langle u_0, \phi_1\rangle^2+\lambda^2l_\sigma^2\Lambda(\theta)\int_0^\infty e^{-\theta t}\E\langle u_t.\phi_1\rangle^2\,\d t
\end{align*}
Here the main observation is that since $2\beta>1$, the function $\Lambda(\theta)$ is bounded. So that for any fixed $\theta>0$ there exists $\lambda_u$ large enough so that for $\lambda\geq \lambda_u$, the above inequality yields
\begin{align*}
\int_0^\infty e^{-\theta t}\E\langle u_t, \phi_1\rangle^2\,\d t&\gtrsim  \Lambda(\theta) \langle u_0, \phi_1\rangle^2+2\int_0^\infty e^{-\theta t}\E\langle u_t, \phi_1\rangle^2\,\d t.
\end{align*}
We now use similar arguments as in the proof of the previous theorem to finish the first part of the current theorem. For the second part, we look at the mild formulation given by \eqref{mild:dirichlet2} which yields
\begin{align*}
\E|u_t(x)|^2&=
|(\mathcal{P}_Du)_t(x)|^2+ \lambda^2 \int_0^L\int_0^t G_D(t-s,\,x,\,y)^2\E|\sigma(u_s(y))|^2\d s\,\d y\\
&\lesssim |(\mathcal{P}_Du)_t(x)|^2+ \lambda^2 L_\sigma^2\int_0^L\int_0^t G_D(t-s,\,x,\,y)^2\E|u_s(y)|^2\d s\,\d y\\
&:=I_1+I_2.
\end{align*}
Since we are assuming that the initial function is bounded above by a constant, the first term $I_1$ is also bounded above by a constant. We now bound $I_2$ as follows. Using the assumption on $\sigma$, we have 
\begin{align*}
I_2&\leq \lambda^2 L_\sigma^2\sup_{0<t<\infty}\sup_{x\in [0,\,L]} \E|u_t(x)|^2\int_0^L\int_0^t G_D(t-s,\,x,\,y)^2\,\d s\,\d y.
\end{align*}
We now use the fact that
\begin{align*}
\int_0^L\int_0^t G_D(t-s,\,x,\,y)^2\,\d s\,\d y\lesssim \int_0^t\sum_{n=0}^\infty E_\beta(-\lambda_n(t-s)^\beta)^2\,\d s.
\end{align*}
Since $\beta\in (\frac{1}{2}, 1)$,  the right hand side of the above display is bounded by a constant and we can choose $\lambda_l$ small enough so that for all $\lambda\leq \lambda_l$, the above estimates yield 
\begin{align*}
\sup_{0<t<\infty}\sup_{x\in [0,\,L]} \E|u_t(x)|^2\lesssim 1+\frac{1}{2}\sup_{0<t<\infty}\sup_{x\in [0,\,L]} \E|u_t(x)|^2.
\end{align*}
This completes the proof of the theorem.
\end{proof}
\section{Proof of Theorems \ref{Theo3} and \ref{Theo4}}
By a slight abuse of notation, we are now going to explicitly indicate the dependence of the solution on $\beta$, we therefore call $u_t^{(\beta)}$, the unique solution to \eqref{mild:dirichlet2} and $u_t$, the solution to \eqref{mild:dirichlet}.

\begin{proof}[Proof of Theorem \ref{Theo3}]
From the mild formulation of the solutions, we have
\begin{align*}
u_t(x)-u_t^{(\beta)}(x)&=(\mathcal{P}_Du)_t(x)-(\mathcal{G}_Du)_t(x)+\lambda \int_0^L\int_0^tp_D(t-s,\,x,\,y)\sigma(u_s(y))W(\d s\,\d y)\\
&-\lambda \int_0^L\int_0^tG_D(t-s,\,x,\,y)\sigma(u_s^{(\beta)}(y))W(\d s\,\d y).
\end{align*}
We look at the stochastic terms first and rewrite them as follows
\begin{align*}
\lambda\int_0^L\int_0^t&p_D(t-s,\,x,\,y)\sigma(u_s(y))W(\d s\,\d y)-\lambda\int_0^L\int_0^tG_D(t-s,\,x,\,y)\sigma(u_s^{(\beta)}(y))W(\d s\,\d y)\\
&=\lambda\int_0^L\int_0^t[p_D(t-s,\,x,\,y)-G_D(t-s,\,x,\,y)]\sigma(u_s^{(\beta)}(y))W(\d s\,\d y)\\
&+\lambda\int_0^L\int_0^tp_D(t-s,\,x,\,y)[\sigma(u_s(y))-\sigma(u^{(\beta)}_s(y))]W(\d s\,\d y).
\end{align*}
We now use the above and Burkholder's inequality to see that 
\begin{align*}
\E|u_t(x)-u_t^{(\beta)}(x)|^p&\lesssim |(\mathcal{P}_Du)_t(x)-(\mathcal{G}_Du)_t(x)|^p\\
&+\left[\int_0^t\int_0^L|p_D(t-s,\,x,\,y)-G_D(t-s,\,x,\,y)|^2[\E|\sigma(u_s^{(\beta)}(y))|^p]^{2/p}\,\d y\,\d s\right]^{p/2}\\
&+\left[\int_0^L\int_0^tp_D(t-s,\,x,\,y)^2[\E[\sigma(u_s(y))-\sigma(u^{(\beta)}_s(y))]^p]^{2/p}\d s\,\d y \right]^{p/2}.
\end{align*}
Fix $\theta>0$ so that the above give us
\begin{align*}
e^{-\theta t} \E|u_t(x)-u_t^{(\beta)}(x)|^p&\lesssim e^{-\theta t} |(\mathcal{P}_Du)_t(x)-(\mathcal{G}_Du)_t(x)|^p\\
&+e^{-\theta t}\left[\int_0^t\int_0^L|p_D(t-s,\,x,\,y)-G_D(t-s,\,x,\,y)|^2[\E|\sigma(u_s^{(\beta)}(y))|^p]^{2/p}\,\d y\,\d s\right]^{p/2}\\
&+e^{-\theta t}\left[\int_0^L\int_0^tp_D(t-s,\,x,\,y)^2[\E[\sigma(u_s(y))-\sigma(u^{(\beta)}_s(y))]^p]^{2/p}\d s\,\d y \right]^{p/2}\\
&:=I_1+I_2+I_3.
\end{align*}
We look at the third term first. Using the fact $\sigma$ is globally Lipschitz, we obtain 
\begin{align*}
I_3&:=e^{-\theta t}\left[\int_0^L\int_0^tp_D(t-s,\,x,\,y)^2[\E[\sigma(u_s(y))-\sigma(u^{(\beta)}_s(y))]^p]^{2/p}\d s\,\d y \right]^{p/2}\\
&\leq \sup_{x\in [0,\,L]}e^{-\theta t}\E|u_t(x)-u_t^{(\beta)}(x)|^p\left[\int _0^\infty\int_0^Le^{-2\theta s/p}p_D(s,\,x,\,y)^2\d y\,\d s\right]^{p/2}.
\end{align*}
We can bound the integral on the right hand side of the above display as follows,
\begin{align*}
\int _0^\infty\int_0^Le^{-2\theta s/p}p_D(s,\,x,\,y)^2\d y\,\d s\lesssim \int_0^\infty e^{-2\theta s/p}\frac{1}{s^{1/2}}\,\d s.
\end{align*}
We now fix $\theta>0$ large enough so that 
$$I_3\lesssim \frac{1}{2} \sup_{x\in [0,\,L]}e^{-\theta t}\E|u_t(x)-u_t^{(\beta)}(x)|^p.$$
We now look at $I_2$.  Using Remark \ref{independence} below and the assumption on $\sigma$, we obtain
\begin{align*}
I_2&=e^{-\theta t}\left[\int_0^t\int_0^L|p_D(t-s,\,x,\,y)-G_D(t-s,\,x,\,y)|^2[\E|\sigma(u_s^{(\beta)}(y))|^p]^{2/p}\,\d y\,\d s\right]^{p/2}\\
&\lesssim \sup_{x\in [0,\,L]}e^{-\theta t}\E|u_s^{(\beta)}(x)|^p\left[\int_0^\infty\int_0^L e^{-2\theta s/p}[p_D(s,\,x,\,y)-G_D(s,\,x,\,y)]^2\,\d s\,\d y\right]^{p/2}\\
&\lesssim \left[\int_0^\infty\int_0^L e^{-2\theta s/p}[p_D(s,\,x,\,y)-G_D(s,\,x,\,y)]^2\,\d s\,\d y\right]^{p/2}.
\end{align*}
Combining the estimates above, we obtain 
\begin{align*}
\sup_{x\in[0,\,L]}e^{-\theta t} \E|u_t(x)-u_t^{(\beta)}(x)|^p&\lesssim  e^{-\theta t} |(\mathcal{P}_Du)_t(x)-(\mathcal{G}_Du)_t(x)|^p\\
&+ \left[\int_0^\infty\int_0^L e^{-2\theta s/p}[p_D(s,\,x,\,y)-G_D(s,\,x,\,y)]^2\,\d s\,\d y\right]^{p/2}.
\end{align*}
We now take $\beta\rightarrow 1$ and use Lemma \ref{con-kernel}  as well as its proof to conclude. 
\end{proof} 
\begin{remark}\label{independence}
We remark that for some $\theta$, we can use the bounds on $G_D(t,\,x,\,y)$ to show that 
\begin{equation*}
\sup_{t>0,\,x\in [0,\,L]}e^{-\theta t}\E|u_s^{(\beta)}(x)|^p
\end{equation*}
is bounded above by a constant independent of $\beta$.  Using the following
\begin{align*}
\int_0^\infty \int_0^L e^{-\theta t}G_D(t,\,x,\,y)^2\,\d y\,\d s\lesssim \int_0^\infty e^{-\theta t}G_D(2t,\,x,\,x)\,\d s
\end{align*}
and the bounds on $G_D(t,\,x,\,y)$, we can show that the left hand side is bounded by a quantity independent of $\beta$. The mild formulation together with some computations similar to those used in the above proof, we have
\begin{align*}
\sup_{t>0,\,x\in [0,\,L]}e^{-\theta t}\E|u^{(\beta)}_t(x)|^p\lesssim 1+\sup_{t>0,\,x\in [0,\,L]}e^{-\theta t}\E|u^{(\beta)}_t(x)|^p\left[\int_0^\infty e^{-2\theta t/p}G_D(t,\,x,\,y)^2\,\d y\right]^{p/2}.
\end{align*}
We thus have the stated claim.
\end{remark}

\begin{proposition}\label{difference}
There exists a constant $K$ independent of $\beta$ such that for any $0<\eta<1$,
\begin{align*}
\E|u^{(\beta)}_t(y)-u^{(\beta)}_s(x)|^p\leq K[|y-x|^a+|t-s|^b],
\end{align*} 
for some constant $a,\,b\in (0,\,1).$
\end{proposition}
\begin{proof}
We will look at the spatial difference first. Let $k=y-x$ and similarly to the proof of the above theorem we have  
\begin{align*}
\E|u_t^{(\beta)}(x+k)&-u_t^{(\beta)}(x)|^p\lesssim |(\mathcal{G}_Du_0)_t(x+k)-(\mathcal{G}_Du_0)_t(x)|^p\\
&+\left[\int_0^t \int_0^L [G_D(t-s,\,x+k,\,y)-G_D(t-s,\,x,\,y)]^2[\E|\sigma(u_s(y)))|^p]^{2/p}\,\d s\,\d y\right]^{p/2}\\
&:=I_1+I_2.
\end{align*}
Using \eqref{bound2}, we can bound the $I_2$ as follows
\begin{align*}
I_2\lesssim c_1|k|^{a p},
\end{align*}
where $a$ is some positive constant less than 1.
For the deterministic part $I_1$, we can use
\begin{align*}
|(\mathcal{G}_Du_0)_{t}(x+k)-(\mathcal{G}_Du_0)_t(x)|&\leq \left[\int_0^L |G_D(t,\,x+k,\,y)-G_D(s,\,x,\,y)|^2u_0(y)\,\d y\right]^{1/2}.
\end{align*}
Hence using \eqref{bound2} again, we have
\begin{align*}
I_1 \lesssim c_2|k|^{a p}.
\end{align*}
We now look at the temporal difference. Assume $t\geq s$ and set $t=s+h$.
\begin{align*}
\E&|u_{s+h}^{(\beta)}(x)-u_s^{(\beta)}(x)|^p\lesssim |(\mathcal{G}_Du_0)_{s+h}(x)-(\mathcal{G}_Du_0)_s(x)|^p\\
&+\E\left[\lambda \int_0^{s+h}\int_0^L G_D(s+h-l,\,x,\,y)\sigma(u_l^{\beta}(y))W(\d l,\d y)-\lambda\int_0^s\int_0^LG_D(s-l,\,x,\,y)\sigma(u_l^{\beta}(y)) W(\d l,\d y)\right]^p\\
&:=I_3+I_4.
\end{align*}
We look at $I_4$ first.
\begin{align*}
I_4&\lesssim \E\left[\lambda \int_0^{s}\int_0^L [G_D(s-l,\,x,\,y)-G_D(s+h-l,\,x,\,y)]\sigma(u_l^{\beta}(y))W(\d l,\d y)\right]^p\\
&+\E\left[\lambda \int_s^{s+h}\int_0^L G_D(s+h-l,\,x,\,y)\sigma(u_l^{\beta}(y))W(\d l,\d y)\right]^p\\
&:=I_5+I_6.
\end{align*}
We use \eqref{bound3} together with Burkholder's inquality as we have done earlier to obtain
\begin{align*}
I_5\lesssim c_3|h|^{bp}.
\end{align*}
The final term $I_6$ can be bounded in a similar fashion. Using \eqref{bound3}, we obtain the same bound as in the above display. We combine all these estimates to arrive at our desired result.
\end{proof}
\begin{proof}[Proof of Theorem \ref{Theo4}]
The proof consists of two main parts which we are going to only sketch. For more details, see \cite{Bezdek}. The first part is a consequence of the Theorem \ref{Theo3} and gives us convergence of the finite dimensional distributions. Indeed for any finite number of points $(t_i,\,x_i)\in [0,\,T]\times [-N,\,N]$, we have convergence in probability of 
\begin{align*}
(u_{t_1}^{(\beta)}(x_1),\,u_{t_2}^{(\beta)}(x_2),\cdots,u_{t_l}^{(\beta)}(x_l))\quad \text{to}\quad (u_{t_1}(x_1),\,u_{t_2}(x_2),\cdots,u_{t_l}(x_l))
\end{align*}
as $\beta \rightarrow 1.$ Tightness follows from Proposition \ref{difference}; indeed a similar argument to that used in \cite{Bezdek} show that
\begin{align*}
\lim_{\delta \rightarrow 0}\sup_{\beta\in(0,\,1)}\P \left(\sup_{|x-y|^{a}+|t-s|^b<\delta}|u_s^{(\beta)}(x)-u_t^{(\beta)}(y)|>\epsilon\right)=0
\end{align*}
We therefore have convergence of the finite dimensional distribution as well as tightness. The result therefore follows.
 \end{proof}

\bibliography{Foon}

\def\cprime{$'$}
\begin{thebibliography}{10}

\bibitem{SFMN}
Sunday Asogwa, Mohammud Foondun, Mijena Jebessa, and Erkan Nane.
\newblock Critical parameters for reaction-diffusion equations involving
  space-time fractional derivatives.
\newblock {\em Submitted}.

\bibitem{BLM}
Boris Baeumer, Luks Tomasz, and Mark Meerschaert.
\newblock Space-time fractional dirichlet problems.
\newblock {\em Mathematische Nachrichten}, to appear, 2018.

\bibitem{Bezdek}
Pavel Bezdek.
\newblock On the weak convergence of stochastic heat equation with colored
  noise.
\newblock {\em Stochastic Process. Appl.}, 126(9):2860--2875, 2016.

\bibitem{CMN}
Zhen-Qing Chen, Mark Meerschaert, and Erkan Nane.
\newblock Space--time fractional diffusion on bounded domains.
\newblock {\em J. Math. Anal. Appl.}, 393:479--488, 2012.

\bibitem{minicourse}
Robert Dalang, Davar Khoshnevisan, Carl Mueller, David Nualart, and Yimin Xiao.
\newblock {\em A minicourse on stochastic partial differential equations},
  volume 1962 of {\em Lecture Notes in Mathematics}.
\newblock Springer-Verlag, Berlin, 2009.

\bibitem{FK}
Mohammud Foondun and Davar Khoshnevisan.
\newblock Intermittence and nonlinear parabolic stochastic partial differential
  equations.
\newblock {\em Electron. J. Probab.}, 14:no. 21, 548--568, 2009.

\bibitem{FoonNual}
Mohammud Foondun and Eulalia Nualart.
\newblock On the behaviour of stochastic heat equations on bounded domains.
\newblock {\em ALEA Lat. Am. J. Probab. Math. Stat.}, 12(2):551--571, 2015.

\bibitem{kra}
Alexander~M. Kr{\"a}geloh.
\newblock Two families of functions related to the fractional powers of
  generators of strongly continuous contraction semigroups.
\newblock {\em J. Math. Anal. Appl.}, 283(459-467), 2003.

\bibitem{Eu}
Eulalia Nualart.
\newblock Moment bounds for some fractional stochastic heat equations on the
  ball.
\newblock {\em Electron. J. Probab.}, 23(41), 2018.

\bibitem{Simon}
Thomas Simon.
\newblock Comparing fr{\'e}chet and positive stable laws.
\newblock {\em Electron. J. Probab.}, 19(16):1--25, 2014.

\bibitem{Walsh}
John~B. Walsh.
\newblock An {I}ntroduction to {S}tochastic {P}artial {D}ifferential
  {E}quations.
\newblock In {\em \'{E}cole d'\'et\'e de {P}robabilit\'es de {S}aint-{F}lour,
  {XIV}---1984}, volume 1180 of {\em Lecture Notes in Math.}, pages 265--439.
  Springer, Berlin, 1986.

\bibitem{Bin}
Bin Xie.
\newblock Some effects of the noise intensity upon non-linear stochastic heat
  equations on [0, 1].
\newblock {\em Stochastic Process. Appl.}, 126(4):1184--1205, 2016.

\end{thebibliography}
\end{document}